\documentclass{article}
\usepackage{amsmath}
\usepackage{amsthm}
\usepackage{amssymb}
\usepackage{float}
\usepackage{tikz}
\usepackage{xcolor}

\newtheorem{theorem}{Theorem}
\newtheorem{lemma}[theorem]{Lemma}
\newtheorem{proposition}[theorem]{Proposition}

\newcommand*\modn[1]{\overline{#1}}

\def\im{\mathop{\mathrm{Im}}\nolimits} 
\def\dom{\mathop{\mathrm{Dom}}\nolimits}
\def\ker{\mathop{\mathrm{Ker}}\nolimits}

\title{Groups of permutations preserving orientation (parity) of subsets of a fixed size, and related monoids}
\author{V\'\i tor Fernandes\footnote{This work is funded by national funds through the FCT - Funda\c c\~ao para a Ci\^encia e a Tecnologia, I.P., under the scope of the (Center for Mathematics and Applications) projects UIDB/00297/2020 (https://doi.org/10.54499/UIDB/00297/2020) and UIDP/00297/2020 (https://doi.org/10.54499/UIDP/00297/2020).}, Alexei Vernitski}

\begin{document}

\maketitle

\begin{abstract}
We study permutations on $n$ elements preserving orientation (parity) of every subset of size $k$. We describe all groups of these permutations. Unexpectedly, these groups (except for some special cases) are either trivial, cyclic or dihedral. In this context, we define and study monoids generalizing monoids of order-preserving mappings and monoids of orientation-preserving mappings. 
\end{abstract}

\section{Introduction} \label{sec:intro}

Fix a positive integer $n$. Denote $0, 1, \dots, n-1$ by $[n]$. Elements of $[n]$ will be denoted by letters $a, b, c$. Mappings on $[n]$ will be denoted by lower-case Greek letters. 
Denote the identity mapping on $[n]$ by $\iota_n$.
For a given $k$, denote the cyclic shift $a \mapsto (a+1) \; (\mathrm{mod} \; n)$ by $\sigma_n$. Denote the cyclic group generated by $\sigma_n$ by $Z_n$. Denote the permutation $a \mapsto (n-1)-a$ reversing the order of elements in $[n]$ by $\rho_n$. Denote the dihedral group generated by $\sigma_n$ and $\rho_n$ by $D_n$. 



A partial mapping $\alpha$ on $[n]$ is called \textit{order-preserving}  
[respectively, \textit{order-reversing}] 
if $x < y$ implies $\alpha(x)\leq \alpha(y)$ 
[respectively, $\alpha(x)\geq \alpha(y)$]  
for all $x,y \in \dom\alpha$.  
An order-preserving or order-reversing partial mapping is also called \textit{monotone}.  The set of order-preserving mappings is closed with respect to composition and contains $\iota_n$, hence, it is a monoid. This monoid is one of the most frequently studied semigroups in semigroup theory; its research originated in \cite{aizenstat1962defining,howie1971products,gomes1992ranks,higgins1993combinatorial,vernitskii1995proof} and led to hundreds of other publications.

One of the ways in which semigroup theorists have generalized the monoid of order-preserving mappings is the following. 
Let $s=(a_1,a_2,\ldots,a_t)$
be a sequence of $t$ ($t\geq0$) elements of $[n]$. 
We say that $s$ is \textit{cyclic} 
[respectively, \textit{anti-cyclic}] if there
exists no more than one index $i\in\{1,\ldots,t\}$ such that
$a_i>a_{i+1}$ [respectively, $a_i<a_{i+1}$],
where $a_{t+1}$ denotes $a_1$.
Notice that the sequence $s$ is cyclic
[respectively, anti-cyclic] if and only if $s$ is empty or there exists
$i\in\{0,1,\ldots,t-1\}$ such that 
$a_{i+1}\leq a_{i+2}\leq \cdots\leq a_t\leq a_1\leq \cdots\leq a_i $ 
[respectively, $a_{i+1}\geq a_{i+2}\geq \cdots\geq a_t\geq a_1\geq \cdots\geq a_i $] (the index
$i\in\{0,1,\ldots,t-1\}$ is unique unless $s$ is constant and
$t\geq2$). We say that $s$ is \textit{oriented} if $s$ is cyclic or $s$ is anti-cyclic. 
Given a partial mapping $\alpha$ on $[n]$ such that
$\dom\alpha=\{a_1<\cdots<a_t\}$, we 
say that $\alpha$ is \textit{orientation-preserving} 
[respectively, \textit{orientation-reversing}, \textit{oriented}] if the sequence of its images
$(\alpha(a_1),\ldots,\alpha(a_t))$ is cyclic [respectively, anti-cyclic, oriented]. The set of orientation-preserving mappings is a monoid. It was introduced 25 years ago \cite{catarino1999monoid,fernandes2000monoid} and has been repeatedly studied since then, some relevant recent publications being \cite{higgins2022orientation,fernandes2023monoid,fernandes2023CKMS}. Unless the rank of $\alpha$ is $2$, being orientation-preserving can be expressed in a language of triples of elements of $[n]$, namely, $\alpha$ is orientation-preserving if and only if from $a < b < c$ it follows that $\alpha(a) \le \alpha (b) \le \alpha (c)$ or $\alpha (b) \le \alpha (c) \le \alpha(a)$ or $\alpha(c) \le \alpha (a) \le \alpha (b)$, see \cite{higgins2022orientation, fernandes2023CKMS, higgins2024correction, fernandes2024CKMS}.

As one compares the monoid of order-preserving mappings and the monoid of orientation-preserving mappings, one can ask the following questions. 

\begin{itemize}
\item The monoid of order-preserving mappings is defined using a condition on pairs of elements of $[n]$, and the mappings in the monoid of orientation-preserving mappings can be described (with the exception of one special case) using a condition on triples of elements of $[n]$. Is it possible to generalize these conditions to $t$-tuples of elements of $[n]$?
\item The only permutation belonging to the monoid of order-preserving mappings is $\iota_n$. In other words, the set of permutations belonging to the monoid of order-preserving mappings is the one-element group $\{\iota_n\}$. As to the monoid of orientation-preserving mappings, the set of permutations belonging to it is the cyclic group $Z_n$. How will this generalize if we consider conditions on $t$-tuples of elements of $[n]$?
\end{itemize}

The aim of this article is to answer these questions. For one-to-one mappings, we introduce a property which generalizes the properties of being order-preserving and orientation-preserving. We successfully describe all groups of permutations satisfying this generalized property. Then we define and study monoids that can be seen as generalizations of monoids of order-preserving mappings and orientation-preserving mappings.

\section{The generalized definition for groups} \label{sec:definition-groups}

Consider a one-to-one partial mapping $\alpha$ on $[n]$ whose domain and image both have size $t$. By an inversion we mean a pair $a, b$ such that $a < b$ and $\alpha(a) > \alpha(b)$. We will say that $\alpha$ is even (odd) if the number of its inversions is even (odd). It is useful to note that if $t=n$ then this definition of being even (odd) coincides with the usual definition of an even (odd) permutation. 

\begin{lemma} \label{lem:product-is-even}
Consider one-to-one partial mappings $\alpha$ and $\beta$ on $[n]$ such that the domain of $\beta$ coincides with the image of $\alpha$. If both $\alpha$ and $\beta$ are even then the composition $\beta(\alpha(\cdot))$ is even.
\end{lemma}
\begin{proof}
We will denote the size of the domain and the image by $t$, as in the definition above. Define a matrix $M$ of a one-to-one partial mapping $\mu$ as follows. The size of $M$ is $t \times t$. Suppose the domain of $\mu$ consists of elements $d_1 < \dots < d_t$ and the image of $\mu$ consists of elements $e_1 < \dots < e_t$. The entry in $M$ at the position $i, j$ is $1$ if $\mu(d_i) = e_j$ and $0$ otherwise. 

It is a well-known fact that a permutation is odd (even) if and only if the corresponding permutation matrix has the determinant equal to $-1$ (equal to $1$). This fact obviously generalizes to the matrices defined above; namely, $\mu$ is odd (even) if and only if the determinant of $M$ is equal to $-1$ (equal to $1$). 

Now notice that the matrix of $\beta(\alpha(\cdot))$ is the product of the matrix of $\alpha$ and the matrix of $\beta$. Hence, the result follows.
\end{proof}

We define $\Gamma_n^t$, where $t = 2, \dots, n$, as the set consisting of those permutations on $[n]$ whose restriction to every $t$-element subset of $[n]$ is even. Due to Lemma \ref{lem:product-is-even}, the set $\Gamma_n^t$ is a group. 


The groups $\{\iota_n\}$ and $Z_n$ that arose in the previous section in the context of monoids of order-preserving mappings and orientation-preserving mappings make a reappearance in the context fo this new definition; indeed, $\Gamma_n^2 = \{\iota_n\}$ and $\Gamma_n^3 = Z_n$. Apart from these two groups, it is easy to spot yet another famous group among groups $\Gamma_n^t$; this is $\Gamma_n^n$, which is the alternating group $A_n$. Speaking of the size of these groups, $|\Gamma_n^2| = 1$, $|\Gamma_n^3| = n$ and $|\Gamma_n^n| = \frac{1}{2}\cdot n!$. What is the size $|\Gamma_n^t|$ if the value of $t$ is in the range between $3$ and $n$? Examples above might be seen as suggesting that the size $|\Gamma_n^t|$ increases monotonically as a function of $t$. However, the real picture is different; let us illustrate it by an example of $|\Gamma_{11}^t|$. The table below shows values of $t$ in the top row and values of $|\Gamma_{11}^t|$ in the bottom row.
\begin{table}[H]
    \centering
    \begin{tabular}{|c|c|c|c|c|c|c|c|c|c|}
    \hline
        2 & 3 & 4 & 5 & 6 & 7 & 8 & 9 & 10 & 11 \\
    \hline
        1 & 11 & 2 & 22 & 1 & 11 & 2 & 22 & 43,200 & 19,958,400\\
    \hline
    \end{tabular}
\end{table}

As you can see, the values of $|\Gamma_n^t|$ oscillate between $1$, $n$, $2$ and $2n$ on the whole range $t = 2, \dots, n-2$, and $t = n-1$ and $t = n$ are special cases. Below we prove theorems which explain in detail how these numbers arise.

For completeness, in addition to the groups considered above, one may also consider group $\Gamma_n^1$; since partial mappings with one-element domains cannot have inversions, every permutation belongs to this group, that is, $\Gamma_n^1$ coincides with the symmetric group $S_n$.

\section{Groups $\Gamma_n^t$ for $t = 2, \dots, n-2$}

Let $S$ be a subset of $[n]$. Let $\alpha$ be a permutation on $[n]$. Denote the restriction of $\alpha$ to $S$ by $\alpha \upharpoonright S$. Denote the number of inversions of the restriction of $\alpha$ to $S$ by $\mathrm{inv}(\alpha \upharpoonright S)$. 

In this section we will sometimes use the usual arithmetic, or the arithmetic modulo $n$, or the arithmetic modulo $2$ (or modulo $4$), in a quick succession. When we work in the arithmetic modulo $2$ or $4$, we will always clearly indicate this by using notation $(\mathrm{mod} \; 2)$ or $(\mathrm{mod} \; 4)$. When we work in the arithmetic modulo $n$, we will stress that by using a bar above the potentially ambiguous expression, for example, $\modn{a+1}$. 

In the proof of Lemma \ref{lem:in-D} below, it will be useful to reformulate the fact that a permutation $\alpha$ belongs or does not belong to $D_n$ in the language of inversions. To achieve this aim, we introduce the following definition. Consider $a, b \in [n]$, where $b \neq a$ and $b \neq \modn{a+1}$. We will say that $\alpha(b)$ lies between $\alpha(a)$ and $\alpha(\modn{a+1})$ if either $a < \modn{a+1}$ and $\alpha(b)$ lies in the interval between $\min(\alpha(a), \alpha(\modn{a+1}))$ and $\max(\alpha(a), \alpha(\modn{a+1}))$ or $a > \modn{a+1}$ and $\alpha(b)$ does not lie in the interval between $\min(\alpha(a), \alpha(\modn{a+1}))$ and $\max(\alpha(a), \alpha(\modn{a+1}))$. The above definition is a mouthful, but it is routine to check that it can be expressed in a much simpler way as saying that $\alpha(b)$ lies between $\alpha(a)$ and $\alpha(\modn{a+1})$ if and only if

$$\mathrm{inv}(\alpha \upharpoonright \{a, b\}) \neq \mathrm{inv}(\alpha \upharpoonright \{\modn{a+1}, b\}) \;(\mathrm{mod} \; 2)$$ or, even more conveniently, $$\mathrm{inv}(\alpha \upharpoonright \{a, b\}) = \mathrm{inv}(\alpha \upharpoonright \{\modn{a+1}, b\}) + 1 \;(\mathrm{mod} \; 2).$$  Hence, the following useful observation is true.

\begin{lemma} \label{lem:between}
Let $a \in [n]$. Consider $b_1, \dots, b_r \in [n]$ such that $\alpha(b_i)$ lies between $\alpha(a)$ and $\alpha(\modn{a+1})$. Consider $c_1, \dots, c_s \in [n]$ such that $\alpha(c_j)$ does not lie between $\alpha(a)$ and $\alpha(\modn{a+1})$. Let $S = \{b_1, \dots, b_r; c_1, \dots, c_s\}$. Let $S' = S \cup \{a\}$ and $S'' = S \cup \{\modn{a+1}\}$. Then $\mathrm{inv}(\alpha \upharpoonright S') = \mathrm{inv}(\alpha \upharpoonright S'') + r \;(\mathrm{mod} \; 2)$.
\end{lemma}

\begin{lemma} \label{lem:in-D}
$\Gamma_n^t \subseteq D_n$ for every $t = 2, \dots, n-2$.
\end{lemma}
\begin{proof}
Suppose a permutation $\alpha$ on $[n]$ does not belong to $D_n$. This can be expressed by saying that there is $a$ such that $\alpha(\modn{a+1})$ is neither $\modn{\alpha(a)+1}$ nor $\modn{\alpha(a)-1}$. Hence, there exists at least one $b \in [n]$ such that $\alpha(b)$ lies between $\alpha(a)$ and $\alpha(\modn{a+1})$.
Consider all elements $b_1, \dots, b_p \in [n] \setminus \{a, \modn{a+1} \}$ such that $\alpha(b_i)$ lies between $\alpha(a)$ and $\alpha(\modn{a+1})$. Consider all elements $c_1, \dots, c_q \in [n] \setminus \{a, \modn{a+1} \}$ such that $\alpha(c_j)$ does not lie between $\alpha(a)$ and $\alpha(\modn{a+1})$. Note that $p+q = n-2$ and $p \ge 1$. The number $t-1$ is in the range $1, \dots, n-3$. Represent $t-1$ as a sum of an odd positive integer $r$ in the range $1, \dots, p$ and a non-negative integer $s$ in the range $0, \dots, q$. Like in Lemma \ref{lem:between}, consider $S = \{b_1, \dots, b_r; c_1, \dots, c_s\}$ and define $S' = S \cup \{a\}$ and $S'' = S \cup \{\modn{a+1}\}$. Since $r$ is odd, by Lemma \ref{lem:between}, we have $\mathrm{inv}(\alpha \upharpoonright S') \neq \mathrm{inv}(\alpha \upharpoonright S'') \;(\mathrm{mod} \; 2)$. That is, if the restriction of $\alpha$ on $S'$ is even then the restriction of $\alpha$ on $S''$ is odd, and if the restriction of $\alpha$ on $S''$ is even then the restriction of $\alpha$ on $S'$ is odd.
Recall that for $\alpha$ to be in $\Gamma_n^t$, the restriction of $\alpha$ to every $t$-element subset of $[n]$ should be even. Hence, $\alpha \notin \Gamma_n^t$.
\end{proof}

\begin{lemma} \label{lem:sigma-in-Gamma}
For every $t = 2, \dots, n-2$, we have $\sigma_n \in \Gamma_n^t$ if and only if $t$ is odd.
\end{lemma}
\begin{proof}
Suppose $t$ is odd. Consider a $t$-element subset of $[n]$ consisting of elements $a_1 < \cdots < a_t$. Since $\sigma_n$ is a cyclic shift, we have $\alpha(a_{p+1}) < \cdots < \alpha(a_t) < \alpha(a_1)  < \cdots < \alpha(a_{p})$ for some $p$. Hence, the number of inversions of the restriction of $\alpha$ to $a_1, \dots, a_t$ can be expressed as $p(t-p)$. Since $t$ is odd, this product is even (irrespective of the value of $p$), hence, $\sigma_n \in \Gamma_n^t$.

Now suppose $t$ is even. Recall that $t \le n-2$. Let $S = \{ 0, \dots, t-2 \} \cup \{ n-1 \}$. The elements $0, \dots, t-2 \in [n]$ are mapped to $1, \dots, t$, respectively, by $\sigma_n$. At the same time, $\sigma_n(n-1) = 0$. Hence, $\mathrm{inv}(\sigma_n \upharpoonright S) = t-1$. Since $t-1$ is an odd number, $\sigma_n \notin \Gamma_n^t$.
\end{proof}

\begin{lemma} \label{lem:rho-in-Gamma}
$\rho_n \in \Gamma_n^t$ if an only if $t = 0 \;(\mathrm{mod} \; 4)$ or $t = 1 \;(\mathrm{mod} \; 4)$.
\end{lemma}
\begin{proof}
As $\rho_n$ acts on a $t$-element subset of $[n]$ consisting of elements $a_1 < \cdots < a_t$, it reverses their order, $\alpha(a_t) < \cdots < \alpha(a_1)$. Thus, $\mathrm{inv}(\rho_n \upharpoonright \{ a_1, \dots, a_t \}) = \frac{t(t-1)}{2}$. The number expressed by this fraction is even if and only if $t = 0 \;(\mathrm{mod} \; 4)$ or $t = 1 \;(\mathrm{mod} \; 4)$.
\end{proof}

From Lemmas \ref{lem:in-D}, \ref{lem:sigma-in-Gamma} and \ref{lem:rho-in-Gamma} our main result follows.
\begin{theorem}\label{2ton-2}
For every $t = 2, \dots, n-2$, 
\begin{itemize}
    \item if $t = 2  \;(\mathrm{mod} \; 4)$ then $\Gamma_n^t$ is the one-element group $\{\iota_n\}$;
    \item if $t = 3  \;(\mathrm{mod} \; 4)$ then $\Gamma_n^t$ is the cyclic group $Z_n$;
    \item if $t = 0  \;(\mathrm{mod} \; 4)$ then $\Gamma_n^t$ is the two-element group $\{\rho_n, \iota_n\}$;
    \item if $t = 1  \;(\mathrm{mod} \; 4)$ then $\Gamma_n^t$ is the dihedral group $D_n$.
\end{itemize}
\end{theorem}

\section{Group $\Gamma_n^{n-1}$}

In this section we will use a type of permutations called parity-alternating permutations \cite{tanimoto2010parity,tanimoto2010combinatorics}. A permutation $\alpha$ on $[n]$ is called \textit{parity-alternating} if for every $a = 0, \dots, n-2$ (we do not consider $a=n-1$) we have $\alpha(a) \neq \alpha(a+1)  \;(\mathrm{mod} \; 2)$. Denote the set of all parity-alternating permutations on $[n]$ by $PAP_n$. It is not difficult to find a formula for the size of $PAP_n$, and that has been done in \cite{tanimoto2010parity}. The size $|PAP_n|$ is $2\cdot(\frac{n}{2})!^2$ if $n$ is even or $(\frac{n-1}{2})!\cdot(\frac{n+1}{2})!$ if $n$ is odd. We will prove that every permutation in $\Gamma_n^{n-1}$ is parity-alternating; to be more precise, the permutations belonging to $\Gamma_n^{n-1}$ constitute exactly a half of $PAP_n$; hence, $|\Gamma_n^{n-1}| = \frac{1}{2}\cdot |PAP_n|$.

Each $(n-1)$-element subset of $[n]$ can be expressed as $[n] \setminus \{ a \}$ for some $a \in [n]$; in this section we will denote this set by $S_a$.

\begin{lemma} \label{lem:in-PAP-all-equal}
If $\alpha \in PAP_n$ then $\mathrm{inv}(\alpha \upharpoonright S_a) = \mathrm{inv}(\alpha \upharpoonright S_b) \;(\mathrm{mod} \; 2)$ for any $a, b \in [n]$.
\end{lemma}
\begin{proof}
Consider $\alpha \in PAP_n$. Fix $a \in \{ 0, \dots, n-2 \}$. Since $a \neq n-1$, the number $r$ of elements of $[n]$ lying between $\alpha(a)$ and $\alpha(a+1)$ is $|\alpha(a) - \alpha(a+1)|-1$.
Since $\alpha \in PAP_n$, the difference between $\alpha(a)$ and $\alpha(a+1)$ is odd; hence, $r$ is even. 
By Lemma \ref{lem:between}, we conclude that $\mathrm{inv}(\alpha \upharpoonright S_a) = \mathrm{inv}(\alpha \upharpoonright S_{a+1}) \;(\mathrm{mod} \; 2)$. This equality is true for every $a \in \{ 0, \dots, n-2 \}$. Hence, either for every $a \in [n]$ the number $\mathrm{inv}(\alpha \upharpoonright S_a)$ is even, or for every $a \in [n]$ the number $\mathrm{inv}(\alpha \upharpoonright S_a)$ is odd. 
\end{proof}

Recall that $\Gamma_n^{n-1}$ consists of all permutations on $[n]$ such that for every $a \in [n]$ the size of $\mathrm{inv}(\alpha \upharpoonright S_a)$ is even. It will be useful to consider a set of all permutations on $[n]$ such that for every $a \in [n]$ the number $\mathrm{inv}(\alpha \upharpoonright S_a)$ is odd (instead of even); we will denote this set by $-\Gamma_n^{n-1}$.

\begin{lemma} \label{lem:Gamma-in-PAP}
$PAP_n$ is a disjoint union of $\Gamma_n^{n-1}$ and $-\Gamma_n^{n-1}$.
\end{lemma}
\begin{proof}
Consider $\alpha \in PAP_n$. By Lemma \ref{lem:in-PAP-all-equal}, either for every $a \in [n]$ the number $\mathrm{inv}(\alpha \upharpoonright S_a)$ is even, or for every $a \in [n]$ the number $\mathrm{inv}(\alpha \upharpoonright S_a)$ is odd.
In the former case, $\alpha$ belongs to $\Gamma_n^{n-1}$, and in the latter case, $\alpha$ belongs to $-\Gamma_n^{n-1}$.

Conversely, consider a permutation $\alpha \in \Gamma_n^{n-1}$ or $\alpha \in -\Gamma_n^{n-1}$. Consider $a \in \{ 0, \dots, n-2 \}$. 
By Lemma \ref{lem:between}, $\mathrm{inv}(\alpha \upharpoonright S_a) = \mathrm{inv}(\alpha \upharpoonright S_{a+1}) + r \;(\mathrm{mod} \; 2)$, where $r$ of elements of $[n]$ lying between $\alpha(a)$ and $\alpha(a+1)$.
If $\alpha \in \Gamma_n^{n-1}$, both $\mathrm{inv}(\alpha \upharpoonright S_a)$ and $\mathrm{inv}(\alpha \upharpoonright S_{a+1})$ are even, hence, $r$ is even. If $\alpha \in \Gamma_n^{n-1}$, both $\mathrm{inv}(\alpha \upharpoonright S_a)$ and $\mathrm{inv}(\alpha \upharpoonright S_{a+1})$ are odd, hence, $r$ is even. Hence, the difference between $\alpha(a)$ and $\alpha(a+1)$ is odd, and $\alpha$ is parity-alternating. 
\end{proof}

\begin{theorem}
$|\Gamma_n^{n-1}| = \frac{1}{2}\cdot |PAP_n|$.
\end{theorem}
\begin{proof}

Recall that in the definition of $\Gamma_n^t$ we assume that $t \ge 2$, hence, when we consider $\Gamma_n^{n-1}$, we assume that $n \ge 3$. On the set $PAP_n$, consider an involution $\Xi$ induced by applying the cycle $(1 \; 3)$ before the permutation. That is, let $\xi$ be a permutation on $[n]$ swapping $1$ and $3$. For each $\alpha \in PAP_n$, the mapping $\Xi$ maps $\alpha$ to $\alpha(\xi(\cdot))$. Since $\xi \in PAP_n$ and $PAP_n$ is a group, it is clear that if $\alpha \in PAP_n$ then $\Xi(\alpha) \in PAP_n$; thus, $\Xi$ is, indeed, an involution on $PAP_n$.
Note that since $\alpha(1) \neq \alpha(3)$, the involution has no fixed points, that is, for every permutation $\alpha$ we have $\Xi(\alpha) \neq \alpha$. 

Now we will prove that $\Xi$ always maps an element of $\Gamma_n^{n-1}$ to an element of $-\Gamma_n^{n-1}$ and vice versa. Compare the set $I'$ of inversions of the restriction of $\alpha$ on $[n] \setminus \{ 1 \}$ and the set $I''$ of inversions of the restriction of $\alpha(\xi(\cdot))$ on $[n] \setminus \{ 3 \}$. Since $\xi$ swaps $1$ and $3$, the set $I'$ contains an inversion $2, 3$ if and only if the set $I''$ does not contain an inversion $1, 2$. All other inversions in $I'$ and $I''$ are the same (subject to changing notation from $3$ to $1$ when comparing $I'$ and $I''$). In terms of the number of elements, this means that either $|I'|$ is odd and $|I''|$ is even, or $|I'|$ is even and $|I''|$ is odd. Generalizing this observation using Lemma \ref{lem:in-PAP-all-equal}, we conclude that either $\alpha \in -\Gamma_n^{n-1}$ and $\Xi(\alpha) \in \Gamma_n^{n-1}$, or $\alpha \in \Gamma_n^{n-1}$ and $\Xi(\alpha) \in -\Gamma_n^{n-1}$. 

The existence of the involution $\Xi$ proves the result.
\end{proof}

As one compares the sizes of $\Gamma_n^{n-1}$ and $\Gamma_n^{n}=A_n$ (for example, looking at the table in Section \ref{sec:definition-groups}), one can notice that $|\Gamma_n^{n-1}|$ is much smaller than $|\Gamma_n^{n}|$, and one can be tempted to conjecture that $\Gamma_n^{n-1} \subseteq \Gamma_n^{n}$. This is not true if $n$ is even; indeed, we have $\sigma_n \in \Gamma_n^{n-1}$ and $\sigma_n \notin \Gamma_n^{n}$. However, it is true that $\Gamma_n^{n-1} \subseteq \Gamma_n^{n}$ if $n$ is odd, as the following result shows. 

\begin{proposition}\label{grnew}
If $n$ is odd then $\Gamma_n^{n-1} = PAP_n \cap \Gamma_n^{n}$.
\end{proposition}
\begin{proof}
For any permutation $\alpha$ on $[n]$, obviously there are exactly $\alpha(0)$ elements $i$ in $[n]$ such that 
$\alpha(i) < \alpha(0)$. Hence, by the definition of an inversion, we have $\mathrm{inv}(\alpha) = \mathrm{inv}(\alpha \upharpoonright S_{0}) + \alpha(0)$.

Suppose $\alpha \in \Gamma_n^{n-1}$. Hence, by Lemma \ref{lem:Gamma-in-PAP}, $\alpha \in PAP_n$. Since $n$ is odd and $\alpha \in PAP_n$, we conclude that $\alpha(0)$ is even. By the definition of $\Gamma_n^{n-1}$, $\mathrm{inv}(\alpha \upharpoonright S_{0})$ is even; hence, using $\mathrm{inv}(\alpha) = \mathrm{inv}(\alpha \upharpoonright S_{0}) + \alpha(0)$, we conclude that $\mathrm{inv}(\alpha)$ is even, that is, $\alpha \in \Gamma_n^{n}$.

Conversely, suppose $\alpha \in PAP_n \cap \Gamma_n^{n}$. Since $n$ is odd and $\alpha \in PAP_n$, we conclude that $\alpha(0)$ is even. Since $\alpha \in \Gamma_n^{n}$, $\mathrm{inv}(\alpha)$ is even. Applying the argument in the previous paragraph in the opposite direction, we conclude that $\mathrm{inv}(\alpha \upharpoonright S_{0})$ is even, hence, $\alpha \in \Gamma_n^{n-1}$.
\end{proof}

\section{A generalized definition for monoids}

Now we proceed to considering monoids of mappings which generalize monoids of order-preserving mappings and monoids of orientation-preserving mappings. 

When we speak of a restriction of a mapping $\alpha$ on $[n]$ to a $t$-element subset of $[n]$, we call it a restriction of width $t$. For $t=1,2,\ldots,n$, we define $\Sigma_n^t$ as the monoid consisting of those mappings on $[n]$ whose all injective restrictions of width $t$ are even.

Suppose $S$ is a set of mappings on $[n]$, and let $k = 1, \dots, n$. We will denote by $S(\mathbf{r}=k)$ (by $S(\mathbf{r} \le k)$, by $S(\mathbf{r} \ge k)$) the subset of $S$ consisting of those mappings whose rank is equal to $k$ (is less or equal to $k$, is greater or equal to $k$). 

Recall that $S_n$ ($A_n$, $T_n$) is notation for the symmetric group (alternating group, symmetric semigroup) on the set $[n]$. 
We can immediately make the following easy observations. 

\begin{proposition} \label{prop:monoids-easy-observations}
1) $\Sigma_n^1=T_n$. 
\\2) $\Sigma_n^n=T_n(\mathbf{r} \le n-1)\cup A_n$.
\\3) $\Gamma_n^t=\Sigma_n^t\cap S_n$, that is, $\Gamma_n^t$ is the group of units of $\Sigma_n^t$ for $t=1,2,\ldots,n$. 
\\4) $\Sigma_n^t(\mathbf{r} \le t-1) = T_n(\mathbf{r} \le t-1)$ for $t=2,\ldots,n$. 
\end{proposition}

\section{Monoids $\Sigma_n^2$ and $\Sigma_n^3$}

We denote by $O_n$, $M_n$, $OP_n$ and $OR_n$ the monoids of all order-preserving mappings, of all monotone mappings, of all orientation-preserving mappings and of all oriented mappings on $[n]$, respectively (the definitions of these types of mappings were given in Section \ref{sec:intro}). As we have said in Section \ref{sec:intro}, the monoid $O_n$ has been studied since the 1960s, and the other three $M_n$, $OP_n$ and $OR_n$ for more than two decades; see for example \cite{aizenstat1962defining,gomes1992ranks,catarino1999monoid,mcalister1998,fernandes2002,fernandesetal2005}. 


\begin{proposition}\label{t=2}
$\Sigma_n^2=O_n$.
\end{proposition}
\begin{proof} 
Indeed, one can reword the definition of $O_n$ as saying that a mapping $\alpha$ on $[n]$ belongs to $O_n$ if and only if each injective restriction of $\alpha$ of width $2$ is order-preserving.
\end{proof}

The following fact is easy to prove and well known, $O_n$ having been described in this way for the first time in \cite{aizenstat1962defining}.
\begin{proposition} \label{prop:generators-of-monoids}
The monoid $O_n$ is generated by $O_n(\mathbf{r} = n-1)$ and $\iota_n$. The monoids $M_n$, $OP_n$ and $OR_n$ are generated by $O_n \cup \{\rho_n\}$, $O_n \cup Z_n$ and $O_n\cup D_n$, respectively. In other words, each of these four monoids is generated by its elements of rank greater than or equal to $n-1$.
\end{proposition}

\begin{proposition} \label{prop:embedding-of-classical-monoids}
For all $t=2,\ldots,n$, $O_n\subseteq\Sigma_n^t$. Also, for all $t = 2, \dots, n-2$,
\begin{itemize} 
    \item if $t = 3  \;(\mathrm{mod} \; 4)$ then $OP_n\subseteq\Sigma_n^t$;
    \item if $t = 0  \;(\mathrm{mod} \; 4)$ then $M_n\subseteq\Sigma_n^t$;
    \item if $t = 1  \;(\mathrm{mod} \; 4)$ then $OR_n\subseteq\Sigma_n^t$. 
\end{itemize}
\end{proposition}
\begin{proof} 
It follows from the definitions of $O_n$ and $\Sigma_n^t$ that $O_n\subseteq\Sigma_n^t$. Combining $O_n\subseteq\Sigma_n^t$ with  Proposition \ref{prop:generators-of-monoids} and Theorem \ref{2ton-2}, we obtain the results in the bullet points.
\end{proof}

\begin{proposition}\label{t=3}
$\Sigma_n^3 = T_n(\mathbf{r} \le 2) \cup OP_n$.
\end{proposition}
\begin{proof}
We have $\Sigma_n^3 \supseteq T_n(\mathbf{r} \le 2)$ by Proposition \ref{prop:monoids-easy-observations}, part 4). We have $\Sigma_n^3 \supseteq OP_n$ by Proposition \ref{prop:embedding-of-classical-monoids}. 

Now consider $\alpha \in \Sigma_n^3$. If the rank of $\alpha$ is $1$ or $2$ then $\alpha \in T_n(\mathbf{r} \le 2)$. Now suppose $\alpha$ has rank greater than or equal to $3$. Let $T=\{i<j<k\}\subseteq [n]$ be such that $\alpha \upharpoonright T$ is injective and suppose that $\alpha(T)=\{a<b<c\}$. Then $\alpha \upharpoonright T$ is presented by one of the following 6 cases.
$$
\left(\begin{smallmatrix}i&j&k\\a&b&c\end{smallmatrix}\right), 
\left(\begin{smallmatrix}i&j&k\\a&c&b\end{smallmatrix}\right),
\left(\begin{smallmatrix}i&j&k\\b&a&c\end{smallmatrix}\right),
\left(\begin{smallmatrix}i&j&k\\c&a&b\end{smallmatrix}\right),
\left(\begin{smallmatrix}i&j&k\\b&c&a\end{smallmatrix}\right),
\left(\begin{smallmatrix}i&j&k\\c&b&a\end{smallmatrix}\right)
$$

Inspecting these cases, we notice that $\alpha \upharpoonright T$ is orientation-preserving if and only if $\alpha \upharpoonright T$ is even, namely, in the following 3 cases.
$$
\left(\begin{smallmatrix}i&j&k\\a&b&c\end{smallmatrix}\right), 
\left(\begin{smallmatrix}i&j&k\\c&a&b\end{smallmatrix}\right),
\left(\begin{smallmatrix}i&j&k\\b&c&a\end{smallmatrix}\right)
$$

It is well known that a mapping $\alpha$ on $[n]$ with rank greater than or equal to $3$ belongs to $OP_n$ if and only if each restriction of $\alpha$ of width $3$ is orientation-preserving \cite{higgins2022orientation,fernandes2023CKMS,fernandes2024CKMS, higgins2024correction, levimitchell2006}. Let us reformulate this result. It is easy to notice that if a partial mapping has a domain of size $3$ and is not injective then it is orientation-preserving. Hence it follows that a mapping $\alpha$ on $[n]$ with rank greater than or equal to $3$ belongs to $OP_n$ if and only if each injective restriction of $\alpha$ of width $3$ is orientation-preserving. 

Combining the observations in the last two paragraphs, we conclude that if $\alpha \in \Sigma_n^3$ and $\alpha$ has rank greater than or equal to $3$ then $\alpha \in OP_n$.
\end{proof}

Comparing Propositions \ref{t=2} and \ref{t=3}, one can ask whether $\Sigma_n^3= OP_n$. The answer is that for $n\geq4$, $OP_n\subsetneqq\Sigma_n^3$; for example,   
$\left(\begin{smallmatrix}0&1&2&3&\cdots&n-1\\0&1&0&1&\cdots&1\end{smallmatrix}\right)\in\Sigma_n^3\setminus OP_n$. This example frequently features in the study of $OP_n$, including \cite{higgins2022orientation,fernandes2023CKMS,fernandes2024CKMS, higgins2024correction}.


Comparing Propositions \ref{t=3} and \ref{prop:embedding-of-classical-monoids}, one might conjecture that $\Sigma_n^4=T_n(\mathbf{r} \le 3)\cup M_n$ and $\Sigma_n^5=T_n(\mathbf{r} \le 4)\cup OR_n$. 
Both these equalities are false. For instance,  for $n\geq4$, 
$\alpha=\left(\begin{smallmatrix}0&\cdots&n-4&n-3&n-2&n-1\\2&\cdots&2&3&0&1\end{smallmatrix}\right)\in\Sigma_n^4$ 
but $\alpha\not\in T_n(\mathbf{r} \le 3)$ and $\alpha\not\in M_n$; and, for $n\geq 5$,  
$\beta=\left(\begin{smallmatrix}0&\cdots&n-5&n-4&n-3&n-2&n-1\\3&\cdots&3&4&1&2&0\end{smallmatrix}\right)\in\Sigma_n^5$ 
but $\beta\not\in T_n(\mathbf{r} \le 4)$ and $\beta\not\in OR_n$. 

\section{Describing monotone and oriented mappings}


Proposition \ref{prop:monoids-easy-observations}, part 4), describes $\Sigma_n^t(\mathbf{r} \le t-1)$ as being equal to $T_n(\mathbf{r} \le t-1)$. In the following section we will find out more about `the upper half' of $\Sigma_n^t$, namely, $\Sigma_n^t(\mathbf{r} \ge t+2)$. In this section we begin this discussion by proving some technical facts concerning semigroups $O_n, M_n, OP_n, OR_n$.


It is known that a mapping $\alpha$ on $[n]$ belongs to $OR_n$ if and only if each restriction  of $\alpha$ of width $4$ is oriented \cite{fernandes2023CKMS, higgins2022orientation}. The following statement reformulates this result in the language of injective restrictions of width $4$, and the proof is unexpectedly fiddly.


\begin{proposition}\label{w4}
Let $\alpha$ be a mapping on $[n]$ with rank greater than or equal to $4$. Then 
$\alpha\in OR_n$ if and only if each injective restriction of $\alpha$ of width $4$ is oriented. 
\end{proposition}
\begin{proof}
A restriction of an oriented mapping is an oriented partial mapping. So, it remains to prove the converse implication. 

Suppose each injective restriction of $\alpha$ of width $4$ is oriented but $\alpha$ is not oriented. 
Then there exists a non-injective restriction $\beta$ of $\alpha$ of width $4$ which is not oriented. 
Hence, $\beta$ has rank $2$ or $3$. 

Suppose that $\beta$ has rank $3$ and let $\dom\beta=\{i<j<k<\ell\}$ and $\im\beta=\{a<b<c\}$. 
Then the sequence of images of $\beta$ is one of the following $12$ cases: 
$(a,b,a,c)$, $(a,b,c,b)$, $(a,c,a,b)$, $(a,c,b,c)$, 
$(b,a,c,a)$, $(b,c,b,a)$, $(b,a,b,c)$, $(b,c,a,c)$, 
$(c,a,b,a)$, $(c,b,a,b)$, $(c,b,c,a)$, $(c,a,c,b)$. 
By multiplying by oriented permutations, i.e. by elements of $D_n$ (notice that $\alpha$ is oriented if and only if $\xi\alpha$ is oriented for all $\xi\in D_n$), we can reduce these $12$ cases to the following two cases, $(a,b,a,c)$ and $(a,b,c,b)$. 

To save space, below we present how to deal with the case $(a,b,a,c)$. For the other case, a similar argument can be given.  

Since the rank of $\alpha$ is at least $4$, there exists $m\in [n]$ such that the restriction of $\alpha$ to  
$\{i,j,k,\ell,m\}$ has rank $4$. Let $d=\alpha(m)$. 
We have $5$ possible cases, $m<i<j<k<\ell$, $i<m<j<k<\ell$, $i<j<m<k<\ell$, $i<j<k<m<\ell$ and $i<j<k<\ell<m$. 
To save space, below we present how to deal with the case $m<i<j<k<\ell$. For the others cases, one can use similar arguments. 

Let us consider the following two injective restrictions of $\alpha$ of width $4$: 
$\beta_1=\left(\begin{smallmatrix}m&i&j&\ell\\d&a&b&c\end{smallmatrix}\right)$ and 
$\beta_2=\left(\begin{smallmatrix}m&j&k&\ell\\d&b&a&c\end{smallmatrix}\right)$. 
Since $\beta_1$ is oriented, we have either $d<a<b<c$ or $a<b<c<d$; 
since $\beta_2$ is oriented, we have $c>d>b>a$.
As you can see, there is a contradiction between these inequalities, and this shows that this case cannot occur. 

Now suppose that $\beta$ has rank $2$. If $\im\beta=\{a<b\}$ then the sequence of images of $\beta$ is $(a,b,a,b)$ or 
$(b,a,b,a)$. As above, by multiplying by an oriented permutation, we can reduce these $2$ cases to one case $(a,b,a,b)$. Proceeding as above, consider a restriction of $\alpha$ of width $5$ and rank $4$ extending $\beta$. We have $15$ possible cases, each of which, with an argument similar to the one above, can be proved to be impossible. 
\end{proof}



The following theorem is a useful summary which lists the above result with several other similar results.  

\begin{theorem}\label{transversals}
Let $\alpha$ be a mapping on $[n]$. 
\\1) If the rank of $\alpha$ is at least $4$ then $\alpha\in OR_n$ if and only if each injective restriction of $\alpha$ of width $4$ is oriented.
\\2) If the rank of $\alpha$ is at least $4$ then $\alpha\in M_n$ if and only if each injective restriction of $\alpha$ of width $4$ is monotone.
\\3) If the rank of $\alpha$ is at least $3$ then $\alpha\in OP_n$ if and only if each injective restriction of $\alpha$ of width $3$ is orientation-preserving.
\\4) If the rank of $\alpha$ is at least $1$ then $\alpha\in O_n$ if and only if each injective restriction of $\alpha$ of width $2$ is order-preserving.
\end{theorem}
\begin{proof}
1) is proved in Proposition \ref{w4}. 2) follows from 1) if we notice that monotone mappings are oriented mappings. 3) was deduced as a part of the proof of Proposition \ref{t=3}. 4) is a simple observation featuring in the proof of Proposition \ref{t=2}.
\end{proof}

\section{How the monoids embed into one another}

It is natural to ask if some of the monoids $\Sigma_n^2,\ldots,\Sigma_n^n$ are equal to one another, like some of the groups $\Gamma_n^2,\ldots,\Gamma_n^n$ are equal to one other, as described in Theorem \ref{2ton-2}. In this section we will show that the monoids $\Sigma_n^2,\ldots,\Sigma_n^n$ are not equal to one another, but some of them are submonoids of others, in a periodic fashion somewhat similar to Theorem \ref{2ton-2}. 

\begin{lemma}\label{winclusion}
If $p,q\in\{2,\ldots,n\}$ are such that $\Sigma_n^p\subseteq\Sigma_n^q$ then $p\leq q$. 
\end{lemma}
\begin{proof}
Suppose $\Sigma_n^p\subseteq\Sigma_n^q$. Assume that $q<p$ (this assumption will lead to a contradiction). 
Consider $\alpha=\left(\begin{smallmatrix}0&1&2&\cdots&p-2&p-1&\cdots&n-1\\1&0&2&\cdots&p-2&p-2&\cdots&p-2\end{smallmatrix}\right)$. 
As to $\alpha \upharpoonright {\{0,1,\ldots,q\}}$, it is injective (recall that $q\leq p-1$) and odd (it has only one inversion). Hence, $\alpha\not\in\Sigma_n^q$. At the same time, $\alpha\in T_n(\mathbf{r} \leq p-1)\subseteq\Sigma_n^p\subseteq\Sigma_n^q$. This contradiction proves that $p\leq q$.  
\end{proof}

From Lemma \ref{winclusion} it follows also that 
$\Sigma_n^p\subsetneqq\Sigma_n^q$ implies $p < q$, and $\Sigma_n^p = \Sigma_n^q$ if and only if $p = q$, for $p,q\in\{2,\ldots,n\}$. Hence, unlike the groups $\Gamma_n^2,\ldots,\Gamma_n^{n-2}$, one cannot find two monoids in the list $\Sigma_n^2,\ldots,\Sigma_n^{n-2}$ which are equal to one another. 


It is easy to show that the next lemma follows from Theorem \ref{2ton-2}, Proposition \ref{grnew} (and the observation before it) and the following properties: 
\begin{itemize}
    \item $\sigma_n\in A_n$ if and only if $n$ is odd;
    \item $\rho_n\in A_n$ if and only if $n = 0  \;(\mathrm{mod} \; 4)$ or $n = 1  \;(\mathrm{mod} \; 4)$; 
    \item $\sigma_n\in \Gamma_n^{n-1}$ if and only if $n$ is even;
    \item $\rho_n\in \Gamma_n^{n-1}$ if and only if $n = 1  \;(\mathrm{mod} \; 4)$ or $n = 2  \;(\mathrm{mod} \; 4)$. 
\end{itemize}

\begin{lemma}\label{ginclusion}
Let $p,q\in\{2,\ldots,n\}$. We have $\Gamma_n^p\subseteq\Gamma_n^q$ if and only if 
$p = 2  \;(\mathrm{mod} \; 4)$ or $q = 1  \;(\mathrm{mod} \; 4)$ or $p = q  \;(\mathrm{mod} \; 4)$. 
\end{lemma}


\begin{lemma} \label{inclusion}
If $p,q\in\{2,\ldots,n\}$ are such that $\Sigma_n^p\subseteq\Sigma_n^q$ then $p\leq q$ and 
either $p = 2  \;(\mathrm{mod} \; 4)$ or $q = 1  \;(\mathrm{mod} \; 4)$ or $p = q  \;(\mathrm{mod} \; 4)$. 
\end{lemma}
\begin{proof}
Since $\Sigma_n^p\subseteq\Sigma_n^q$ implies $\Gamma_n^p\subseteq\Gamma_n^q$, the result is an immediate consequence of Lemmas \ref{winclusion} and \ref{ginclusion}.
\end{proof}


Our next task is to prove the converse of Lemma \ref{inclusion}; we do it in Lemma \ref{converse}. 






\begin{lemma} \label{converse}
1) Let $2\leq p\leq q\leq k\leq n$. 
If $p = 2  \;(\mathrm{mod} \; 4)$ or $q = 1  \;(\mathrm{mod} \; 4)$ or $p = q  \;(\mathrm{mod} \; 4)$, then 
$\Sigma_n^p(\mathbf{r} = k) \subseteq \Sigma_n^q(\mathbf{r} = k)$. 
\\2) For $t\leq k-2\leq n-2$, 
\begin{itemize}
    \item if $t = 2  \;(\mathrm{mod} \; 4)$ then $\Sigma_n^t(\mathbf{r} = k)\subseteq O_n$; 
    \item if $t = 3  \;(\mathrm{mod} \; 4)$ then $\Sigma_n^t(\mathbf{r} = k)\subseteq OP_n$; 
    \item if $t = 0  \;(\mathrm{mod} \; 4)$ then $\Sigma_n^t(\mathbf{r} = k)\subseteq M_n$; 
    \item if $t = 1  \;(\mathrm{mod} \; 4)$ then $\Sigma_n^t(\mathbf{r} = k)\subseteq OR_n$.
\end{itemize}
3) For $t\leq k-2\leq n-2$, 
\begin{itemize}
    \item if $t = 2  \;(\mathrm{mod} \; 4)$ then $\Sigma_n^t(\mathbf{r} \ge k)  = O_n(\mathbf{r} \ge k)$; 
    \item if $t = 3  \;(\mathrm{mod} \; 4)$ then $\Sigma_n^t(\mathbf{r} \ge k) = OP_n(\mathbf{r} \ge k)$; 
    \item if $t = 0  \;(\mathrm{mod} \; 4)$ then $\Sigma_n^t(\mathbf{r} \ge k) = M_n(\mathbf{r} \ge k)$; 
    \item if $t = 1  \;(\mathrm{mod} \; 4)$ then $\Sigma_n^t(\mathbf{r} \ge k) = OR_n(\mathbf{r} \ge k)$.
\end{itemize}
\end{lemma}
\begin{proof}
Consider $\alpha\in \Sigma_n^t(\mathbf{r} = k)$. Let $X$ be any transversal of $\ker\alpha$. 
Let $\alpha_L$ and $\alpha_R$ be the injective order-preserving partial mappings such that 
$\dom\alpha_L=[k]$, $\im\alpha_L=X$, $\dom\alpha_R=\im\alpha$, $\im\alpha_R=[k]$. 
Consider $\beta=\alpha_L\alpha\alpha_R$ (that is, $\beta$ is the composition $\alpha_R(\alpha(\alpha_L(.)))$. 
Note that $\beta$ is a permutation on $[k]$. 
Let $Y$ be a $t$-element subset of $[k]$. Then we have 
$$
\beta \upharpoonright Y=(\alpha_L \upharpoonright Y)(\alpha \upharpoonright {\alpha_L(Y)})(\alpha_R \upharpoonright {\alpha(\alpha_L(Y))}).
$$
Since $\alpha_L(Y)$ is a $t$-element subset of $[n]$ and $\alpha\in\Sigma_n^t$, 
we conclude that $\alpha \upharpoonright {\alpha_L(Y)}$ is even. 
At the same time, both $\alpha_L \upharpoonright Y$ and $\alpha_R \upharpoonright {\alpha(\alpha_L(Y)}$ are order-preserving injective mappings, hence, have no inversions. 
Therefore, by Lemma \ref{lem:product-is-even}, $\beta \upharpoonright Y$ is even. Applying this argument to every $Y$, that is, every $t$-element subset of $[k]$, we conclude that $\beta\in\Gamma_k^t$. 

Therefore, by Theorem \ref{2ton-2}, we have the following cases.
\begin{itemize}
    \item if $t = 2  \;(\mathrm{mod} \; 4)$ then $\beta=\iota_k$; 
    \item if $t = 3  \;(\mathrm{mod} \; 4)$ then $\beta\in Z_k$; 
    \item if $t = 0  \;(\mathrm{mod} \; 4)$ then $\beta\in\{\rho_k,\iota_k\}$;
    \item if $t = 1  \;(\mathrm{mod} \; 4)$ then $\beta\in D_k$. 
\end{itemize}


Now notice that $\alpha \upharpoonright X=\alpha_L^{-1}\beta\alpha_R^{-1}$. Since $\alpha_L^{-1}$ and $\alpha_R^{-1}$ are order-preserving injective mappings, for $t\leq k-2$ we make the following conclusions. 
\begin{itemize}
    \item if $t = 2  \;(\mathrm{mod} \; 4)$ then $\alpha \upharpoonright X$ is order-preserving; 
    \item if $t = 3  \;(\mathrm{mod} \; 4)$ then $\alpha \upharpoonright X$ is orientation-preserving; 
    \item if $t = 0  \;(\mathrm{mod} \; 4)$ then $\alpha \upharpoonright X$ is monotone; 
    \item if $t = 1  \;(\mathrm{mod} \; 4)$ then $\alpha \upharpoonright X$ is oriented. 
\end{itemize}
Combining these observations with Theorem \ref{transversals}, we obtain part 2) of the lemma. Part 3) follows from part 2) and Proposition \ref{prop:embedding-of-classical-monoids}.


For part 1) of the lemma, consider $p\leq q$ and suppose $\alpha\in \Sigma_n^p(\mathbf{r} = k)$. Consider the same product $\beta=\alpha_L\alpha\alpha_R$ as above, and apply the same argument, with $t = p$. We conclude that $\beta\in\Gamma_k^p$. Hence, and since conditions $p = 2  \;(\mathrm{mod} \; 4)$ or $q = 1  \;(\mathrm{mod} \; 4)$ or $p = q  \;(\mathrm{mod} \; 4)$ are satisfied, by Lemma \ref{ginclusion}, $\beta\in\Gamma_k^q$. As above, notice that $\alpha \upharpoonright X=\alpha_L^{-1}\beta\alpha_R^{-1}$ and recall that $\alpha_L^{-1}$ and $\alpha_R^{-1}$ are order-preserving injective mappings; hence, every injective restriction of $\alpha$ of width $q$ is even. Therefore, $\alpha\in \Sigma_n^q(\mathbf{r} = k)$.
\end{proof}

Combining the results of this section, we obtain the following theorem, illustrated by the following figure. 

\begin{theorem}\label{relations}
Let $p,q\in\{2,\ldots,n\}$. Then, $\Sigma_n^p\subseteq\Sigma_n^q$ if and only if $p\leq q$ and 
either $p = 2  \;(\mathrm{mod} \; 4)$ or $q = 1  \;(\mathrm{mod} \; 4)$ or $p = q  \;(\mathrm{mod} \; 4)$. 
\end{theorem}
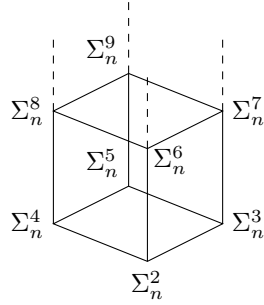
\begin{center}
\begin{tikzpicture}[scale=1.0]
\coordinate (C) at (0.75,0.5); 
\coordinate (A) at (2,0);
\coordinate (B) at (3,0.5);
\coordinate (D) at (1.75,1); 

\coordinate (G) at (0.75,2.0);
\coordinate (E) at (2,1.5); 
\coordinate (F) at (3,2); 
\coordinate (H) at (1.75,2.5); 

\coordinate (G1) at (0.75,3.0); 
\coordinate (E1) at (2,2.5); 
\coordinate (F1) at (3,3); 
\coordinate (H1) at (1.75,3.5); 

\draw (A) -- (B) -- (D) -- (C) -- cycle;
\draw (E) -- (F) -- (H) -- (G) -- cycle;

\draw (A) -- (E);
\draw (B) -- (F);
\draw (C) -- (G);
\draw (D) -- (H);

\draw[dashed] (E) -- (E1);
\draw[dashed] (F) -- (F1);
\draw[dashed] (G) -- (G1);
\draw[dashed] (H) -- (H1);

\node[below] at (A) {$\Sigma_n^2$}; 
\node[right] at (B) {$\Sigma_n^3$}; 
\node[left] at (C) {$\Sigma_n^4$}; 
\node[above left] at (D) {$\Sigma_n^5$}; 

\coordinate (E2) at (2.3,1.35); 
\node at (E2) {$\Sigma_n^6$};
\node[right] at (F) {$\Sigma_n^7$}; 
\node[left] at (G) {$\Sigma_n^8$}; 
\node[above left] at (H) {$\Sigma_n^9$}; 

\end{tikzpicture}
\end{center}


In this paper we consider monoids of full mappings on $[n]$ that extend the groups $\Gamma_n^1,\ldots,\Gamma_n^n$. 
Similar research can be conducted by considering monoids of partial mappings on $[n]$ or inverse monoids of partial permutations on $[n]$. 

\section{Another generalized definition for monoids}

As we reflect on Proposition \ref{prop:generators-of-monoids} and Theorem \ref{relations}, this motivates us to consider the following, alternative family of monoids. For $t=1,2,\ldots,n$, we define $\Delta_n^t$ as the monoid generated by all mappings on $[n]$ whose rank is at least $n-1$ and whose injective restrictions of width $t$ is even; in other words, $\Delta_n^t$ is generated by $\Sigma_n^t(\mathbf{r} \ge n-1)$. 

The groups of units of the new monoids $\Delta_n^t$ are the groups $\Gamma_n^t$, for $t=1,\ldots,n$, in the same way as it was for the family of monoids $\Sigma_n^t$ considered above.
We have $\Delta_n^1=\Sigma_n^1=T_n$ (since $T_n$ is generated by its mappings of rank greater than or equal to $n-1$)  and $\Delta_n^n=\Sigma_n^n=T_n(\mathbf{r} \le n-1)\cup A_n$. 
More interesting is the fact that the monoids $\Delta_n^2,\ldots,\Delta_n^{n-3}$ demonstrate a periodic behavior similar to 
the one described by Theorem \ref{2ton-2} for the groups $\Gamma_n^2,\ldots,\Gamma_n^{n-2}$, as described in the following theorem. 

\begin{theorem}\label{2ton-2mon}
For every $t = 2, \dots, n-3$, 
\begin{itemize}
    \item if $t = 2  \;(\mathrm{mod} \; 4)$ then $\Delta_n^t=O_n$;  
    \item if $t = 3  \;(\mathrm{mod} \; 4)$ then $\Delta_n^t=OP_n$;
    \item if $t = 0  \;(\mathrm{mod} \; 4)$ then $\Delta_n^t=M_n$;
    \item if $t = 1  \;(\mathrm{mod} \; 4)$ then $\Delta_n^t=OR_n$.
\end{itemize}
\end{theorem}
\begin{proof}
This result follows from Proposition \ref{prop:generators-of-monoids} and Theorem \ref{relations}. 
\end{proof}

One can ask if the above theorem generalizes to the case $t = n-2$. For $n=4,5$ we have $\Delta_4^2=O_4$ and $\Delta_5^3=OP_5$, in accordance with the above theorem. However, this is not true for $n\geq6$; for example, we have 
$\alpha=\left(\begin{smallmatrix}0&1&2&3&4&\cdots&n-2&n-1\\2&3&0&1&4&\cdots&n-2&n-2\end{smallmatrix}\right)\in\Delta_n^{n-2}$ 
but $\alpha\not\in OR_n$. 


\bibliographystyle{IEEEtran}
\bibliography{main}

\bigskip 

\noindent{\sc V\'\i tor H. Fernandes},
Center for Mathematics and Applications (NOVA Math)
and Department of Mathematics,
Faculdade de Ci\^encias e Tecnologia,
Universidade Nova de Lisboa,
Monte da Caparica,
2829-516 Caparica,
Portugal;
e-mail: vhf@fct.unl.pt.

\noindent{\sc Alexei Vernitski},
School of Mathematics, Statistics and Actuarial Science,
University of Essex,
Colchester, UK;
e-mail: asvern@essex.ac.uk.

\end{document}